\documentclass[11pt]{article}

\usepackage{amsfonts}
\usepackage{amsxtra}
\usepackage{amsthm}
\usepackage{amssymb}
\usepackage{amsmath,amscd}
\usepackage{supertabular}
\usepackage{pstricks}
\usepackage{pst-plot}
\usepackage{pstricks-add}
\usepackage{pst-eps}
\usepackage{makeidx}
\usepackage{bm}

\newtheorem{theorem}{Theorem}[section]

\newtheorem{remark}[theorem]{Remark}

\DeclareMathOperator\SL{SL}
\DeclareMathOperator\Sp{Sp}
\DeclareMathOperator\PSL{PSL}
\DeclareMathOperator\U{U}
\DeclareMathOperator\Or{O}

\DeclareMathOperator\PGO{PGO}

\DeclareMathOperator\PU{PU}

\makeindex

\begin{document}

\title{An Odd Presentation for $W(\mathrm{E}_6)$}

\author
{Gert Heckman and Sander Rieken\\
Radboud University Nijmegen}

\date{\today}

\maketitle

\begin{abstract}
The Weyl group $W(\mathrm{E}_6)$ has an odd presentation due to Christopher
Simons as factor group of the Coxeter group on the Petersen graph
by deflation of the free hexagons. 
The goal of this paper is to give a geometric meaning for this presentation,
coming from the action of $W(\mathrm{E}_6)$ on the moduli space of marked 
maximally real cubic surfaces and its natural tessellation as seen through 
the period map of Allcock, Carlson and Toledo.
\end{abstract}

\section{Introduction}

We denote by $\mathcal{M}(1^n)$ the moduli space of $n$ ordered mutually distinct
points on the complex projective line. If $n=n_1+\cdots+n_r$ is a partition of $n$
with $r\geq4$ parts we denote by $\mathcal{M}(n_1\cdots n_r)$ the moduli space 
of $r$ points on the complex projective line with weights $n_1,\cdots,n_r$
respectively, and to be viewed as part of a suitable compactification of
$\mathcal{M}(1^n)$ by collisions according to the given partition.

The case of $4$ points is classical and very well known. If $z=(z_1,z_2,z_3,z_4)$
represents a point of $\mathcal{M}(1^4)$ then we consider for the elliptic curve
\[ E(z):\;y^2=\prod(x-z_i)\]
with periods (say $z_i$ are all real with $z_1<z_2<z_3<z_4$)
\[ \pi_i(z)=\int_{z_i}^{z_{i+1}}\frac{dx}{y} \] 
resulting in a coarse period isomorphism (by taking the ratio of two consecutive
periods)
\[ \mathcal{M}(1^4)/S_4\longrightarrow\mathbb{H}/\Gamma  \]
of orbifolds. Here $S_n$ is the symmetric group on $n$ objects and $\Gamma$ is the
modular group $\PSL_2(\mathbb{Z})$ acting on the upper half plane 
$\mathbb{H}=\{\tau\in\mathbb{C};\Im{\tau}>0\}$ by fractional linear transformations.
To remove the orbifold nature one observes an underlying fine period isomorphism
\[ \mathcal{M}(1^4)\longrightarrow\mathbb{H}/\Gamma(4) \]
with $\Gamma(4)$ the principal congruence subgroup of $\Gamma$ of level $4$.
Taking the quotient on the left by $S_4$ and on the right by $\Gamma/\Gamma(4)\cong
S_4$ turns the fine period isomorphism into the previous coarse one.

There are two different real loci: either all 4 points are real or 2 points are
real and 2 are complex conjugate. The first component is called the maximal real 
locus. Under the coarse period isomorphism the maximal real locus corresponds to 
the imaginary axis in $\mathbb{H}$ since $\pi_{i+1}/\pi_i$ is purely imaginary, 
while the other real locus corresponds to the unit circle in $\mathbb{H}$. 
The group $\Gamma(4)$ has $6$ cusps and is of genus $0$ meaning that the 
compactification $\overline{\mathbb{H}}/\Gamma(4)$ by filling in the cusps is just 
isomorphic to the complex projective line.
The $6$ cusps are just the vertices of an octahedron permuted transitively by the
rotation group $S_4$ of the octahedron (permuting pairs of opposite faces).
The maximal real locus corresponds to the $12$ edges of the octahedron, while
the other real locus corresponds to the $24$ diagonals in the faces of the
octahedron.

This simple picture allows a beautiful generalization. If $z=(z_1,\cdots,z_6)$
represents a point of $\mathcal{M}(1^6)$ then we consider the curve 
\[ C(z):\;y^3=\prod(x-z_i)\]
which is of genus $4$ by the Hurwitz formula. The Jacobian $J(C(z))$ is a 
principally polarized Abelian variety of dimension $4$ with an endomorphism 
structure by the group ring $\mathbb{Z}[C_3]$ of the cyclic group of order $3$.
The PEL theory of Shimura \cite{Shimura 1963}, \cite{Shimura 1964}, 
\cite{Casselman 1966} gives that these Jacobians in the full moduli space 
$\mathcal{A}_4=\mathbb{H}_4/\Sp_8(\mathbb{Z})$ form an open dense part of 
a ball quotient $\mathbb{B}/\Gamma$ of dimension $3$. More precisely and thanks 
to the work of Deligne and Mostow \cite{Deligne--Mostow 1986} and of Terada 
\cite{Terada 1983} we have a coarse period isomorphism
\[ \mathcal{M}(1^6)/S_6\longrightarrow\mathbb{B}^{\circ}/\Gamma \]
with $\mathbb{B}^{\circ}/\Gamma$ the complement of a Heegner divisor in
a ball quotient $\mathbb{B}/\Gamma$. More explicitly, let 
$\mathcal{E}=\mathbb{Z}+\mathbb{Z}\omega$ with $\omega=(-1+i\sqrt3)/2$ be the ring
of Eistenstein integers and let 
\[ L=\mathcal{E}\otimes\mathbb{Z}^{3,1} \]
be the Lorentzian lattice over $\mathcal{E}$ then it turns out that
the automorphism group $\U(L)$ is a group generated by the hexaflections
(order 6 complex reflections) in norm one vectors. If $e\in L$ is a norm one 
vector then the hexaflection with root $e$ is defined by 
$h_e(l)=l-(\omega^2+1)\langle l,e\rangle e$.
Here $\langle\cdot,\cdot\rangle$ denotes the sesquilinear form on $L$ of 
Lorentzian signature. Let us denote the complement of the mirrors of all
these hexaflections by $\mathbb{B}^{\circ}$.
The main result of Deligne and Mostow in this particular case can be rephrazed 
by the commutative diagram
\[ \begin{CD}
 \mathcal{M}^{\circ} @>>> \mathcal{M} @>>> \overline{\mathcal{M}}^{\mathrm{HM}}  \\
 @VVV                 @VVV           @VVV                                  \\   
 \mathbb{B}^{\circ}/\Gamma @>>> \mathbb{B}/\Gamma @>>>                    
 \overline{\mathbb{B}}^{\mathrm{BB}}/\Gamma                               \\
\end{CD} \]
with $\mathcal{M}^{\circ}$ short for $\mathcal{M}(1^6)/S_6$. The horizontal maps are 
injective and the vertical maps are isomorphisms from the top horizontal line (the
geometric side) to the bottom horizontal line (the arithmetic side). 
The moduli space
\[ \overline{\mathcal{M}}^{\mathrm{HM}}=
\mathrm{Proj}\left(S(S^6\mathbb{C}^2)^{\SL_2(\mathbb{C})}\right) \] 
is the Hilbert--Mumford compactification of $\mathcal{M}^{\circ}$ through
GIT of degree 6 binary forms, which consists of the open stable locus
$\mathcal{M}$ with at most double collisions and the polystable (also called
strictly semistable) locus, a point with two triple collisions.
In the bottom line we have the ball quotient $\mathbb{B}/\Gamma$ with
$\Gamma=\PU(L)$ and its Baily--Borel compactification
\[ \overline{\mathbb{B}}^{\mathrm{BB}}/\Gamma=
\mathrm{Proj}\left(\mathcal{A}(\mathbb{L}^{\times})^{\U(L)}\right)   \]
with $\mathbb{L}^{\times}=\{v\in\mathbb{C}\otimes\mathbb{Z}^{3,1};\langle
v,v\rangle<0\} \longrightarrow\mathbb{B}=\mathbb{P}(\mathbb{L})$ the natural
$\mathbb{C}^{\times}$-bundle and $\mathcal{A}(\mathbb{L}^{\times})^{\U(L)}$ the
algebra of modular forms, graded by weight
(minus the degree, or maybe better by minus degree/$3$ in order to match with the
degree on the geometric side: the center of $\SL_2(\mathbb{C})$ has order $2$ while
the center of $\U(L)$ has order $6$). 

A similar commutative diagram also holds in the case of ordered points, so
with $\mathcal{M}^{\circ}=\mathcal{M}(1^6)/S_6$ replaced by 
$\mathcal{M}_m^{\circ}=\mathcal{M}(1^6)$ and $\mathrm{U}(L)$ 
replaced by the principal congruence subgroup $\mathrm{U}(L)(1-\omega)$.
The subindex $m$ stands for marking. 
This latter group is generated by all triflections in norm one vectors, 
namely by the squares of the previous hexaflections. So we have a commutative diagram
\[ \begin{CD}
 \mathcal{M}_m^{\circ} @>>> \mathcal{M}_m @>>> \overline{\mathcal{M}}_m^{\mathrm{HM}}  \\
 @VVV                 @VVV           @VVV                                  \\   
 \mathbb{B}^{\circ}/\Gamma(1-\omega) @>>> \mathbb{B}/\Gamma(1-\omega) @>>>                    
 \overline{\mathbb{B}}^{\mathrm{BB}}/\Gamma(1-\omega)                               \\
\end{CD} \] 
The group isomorphism $\Gamma/\Gamma(1-\omega)\cong S_6$ explains that the quotient of 
this commutative diagram by this finite group gives back the former commutative diagram.

The real locus in the space $\mathcal{M}(1^6)/S_6$ of degree $6$ binary forms
with nonzero discriminant has $4$ connected components. There are k complex
conjugate pairs and the remaining points $6-2k$ points are real for
$k=0,1,2,3$ respectively. All $6$ points real is called the maximal real locus,
and will be denoted $\mathcal{M}_r^{\circ}=\mathcal{M}_r(1^6)/S_6$. It was shown by 
Yoshida \cite{Yoshida 1998} that we have a similar commutative diagram
\[ \begin{CD}
 \mathcal{M}_r^{\circ} @>>> \mathcal{M}_r @>>> \overline{\mathcal{M}}_r^{\mathrm{HM}}  \\
 @VVV                 @VVV           @VVV                                  \\   
 \mathbb{B}_r^{\circ}/\Gamma @>>> \mathbb{B}_r/\Gamma @>>>                    
 \overline{\mathbb{B}}_r^{\mathrm{BB}}/\Gamma                               \\
\end{CD} \]
with the bar in the upper horizontal line denoting the real Zariski closure of 
the maximal real locus in the GIT compactification, and the bar in the lower
horizontal line denoting the Baily--Borel compactification of $\mathbb{B}_r$.
Here $\mathbb{B}_r$ is the real hyperbolic ball associated to the Lorentzian lattice 
$\mathbb{Z}^{3,1}$. Likewise $\mathbb{B}_r^{\circ}$ is the complement of the mirrors 
in norm one roots in $\mathbb{Z}^{3,1}$ and $\Gamma=\Or^+(\mathbb{Z}^{3,1})$. 

Likewise we have a marked version in the real case with commutative diagram
\[ \begin{CD}
 \mathcal{M}_{rm}^{\circ} @>>> \mathcal{M}_{rm} @>>> 
 \overline{\mathcal{M}}_{rm}^{\mathrm{HM}}  \\
 @VVV                 @VVV           @VVV                                  \\   
 \mathbb{B}_{r}^{\circ}/\Gamma(3) @>>> \mathbb{B}_{r}/\Gamma(3) @>>>                    
 \overline{\mathbb{B}}_{r}^{\mathrm{BB}}/\Gamma(3)                        \\
\end{CD} \]
with $\mathcal{M}_{rm}^{\circ}=\mathcal{M}_r(1^6)$ the moduli space of $6$ distinct
ordered real points and $\Gamma(3)$ the principal congruence subgroup of 
$\Gamma=\Or^+(\mathbb{Z}^{3,1})$ of level $3$. The group isomorphism 
$\Gamma/\Gamma(3)=\mathrm{PGO}_4(3)\cong S_6$ shows that the quotient of this
commutative diagram by $S_6$ gives the previous commutative diagram just as in 
the complex case.

Deliberately we have suppressed the index $n=3$ of the Lorentzian lattice
$\mathbb{Z}^{n,1}$ because there are similar stories to tell for $n=2,3,4$.
The case $n=2$ corresponds to $\mathcal{M}^{\circ}=\mathcal{M}(21^4)/S_4$ and
$\mathcal{M}_m^{\circ}=\mathcal{M}(21^4)$ and is also due to Deligne and Mostow.
The case $n=4$ corresponds to $\mathcal{M}^{\circ}=\mathcal{M}(cs)$ the moduli
space of smooth cubic surfaces and is due to Allcock, Carlson and Toledo
\cite{Allcock--Carlson--Toledo 2002}. A cubic surface $S$ can be obtained by 
blowing up $6$ points in the projective plane and hence $H_2(S,\mathbb{Z})$
with its insection form is just the lattice $\mathbb{Z}^{1,6}$ with natural basis
$l,e_1,\cdots,e_6$ for a line and the exceptional curves. The anticanonical class
$k=3l-\sum e_i$ has norm $3$ and its orthogonal complement is isomorphic to 
minus the root lattice of type $\mathrm{E}_6$. The choice of such an isomorphism
is called a marking of the cubic surface $S$. The Weyl group $W(\mathrm{E}_6)$ 
permutes these markings in a simply transitive manner. We denote by 
$\mathcal{M}_m^{\circ}=\mathcal{M}_m(cs)$ the moduli space of marked smooth cubic
surfaces. The maximal real locus $\mathcal{M}_r^{\circ}=\mathcal{M}_r(cs)$ is 
by definition the moduli space of smooth real cubic surfaces with $27$ real lines, 
and likewise we denote $\mathcal{M}_{rm}^{\circ}=\mathcal{M}_{rm}(cs)$ for the 
marked covering. All four commutative diagrams remain valid in case $n=4$. 
The group isomorphism $\Gamma/\Gamma(3)=\mathrm{PGO}_5(3)\cong W(\mathrm{E}_6)$ 
shows that the quotient of the commutative diagram in the marked case becomes 
the commutative diagram in the unmarked case.

Consider following commutative diagram
\[ \begin{CD}
\mathcal{M}_{rm}^{\circ} @>>> \overline{\mathcal{M}}_{rm} @>>>
\overline{\mathcal{M}}_r            \\
@VVV                             @VVV                      @VVV                     
                     \\  
\mathbb{B}^{\circ}_r/\Gamma(3) @>>> \overline{\mathbb{B}}_r/\Gamma(3) @>>>
\overline{\mathbb{B}}_r/\Gamma  \\
\end{CD} \]
with $\Gamma/\Gamma(3)=\mathrm{PGO}_{n+1}(3)$ the Weyl group
of type $\mathrm{A}_3,\mathrm{A}_5,\mathrm{E}_6$ for $n=2,3,4$ respectively.
The two left horizontal arrows are inclusions and the two right horizontal maps
are quotient maps for the action of $\Gamma/\Gamma(3)$. In fact we shall for the
moment only consider the bottom horizontal line for all $2\leq n\leq7$, 
independently of the modular interpretation for $n\leq4$. 

Fix a connected component of the mirror complement $\mathbb{B}^{\circ}_r$
of norm one roots in $\mathbb{Z}^{n,1}$ and denote by $P$ its closure in 
$\overline{\mathbb{B}}_r$. It is a fundamental domain for the action on
$\overline{\mathbb{B}}_r$ of the subgroup $\Gamma_1$ of 
$\Gamma=\mathrm{O}^+(\mathbb{Z}^{n,1})$ generated by
the reflections in norm one roots. Clearly $\Gamma_1$ is a subgroup 
of the principal congruence subgroup $\Gamma(2)$ of level $2$. It was shown by 
Everitt, Ratcliffe and Tschantz that $\Gamma_1=\Gamma(2)$ if and only if 
$n\leq7$, which will be assumed from now on. The polytope $P$ will be 
called the Gosset polytope, by analogy with the terminolgy of Coxeter
\cite{Coxeter 1940} in case $n=6$. The symmetry group $\Gamma_0$ of $P$ in $\Gamma$ 
is the Coxeter group of type $\mathrm{E}_n$, with $\mathrm{E}_5=\mathrm{D}_5$,
$\mathrm{E}_4=\mathrm{A}_4$, $\mathrm{E}_3=\mathrm{A}_1\sqcup\mathrm{A}_2$
and $\mathrm{E}_2=\mathrm{A}_1$. For $n\geq3$ it permutes the faces of $P$ 
transitively, and a face of $P^n$ is equal to $P^{n-1}$. The ball quotient 
$\overline{\mathbb{B}}_r/\Gamma(3)$ inherits a regular tessellation by 
polytopes $\gamma P$ with $\gamma\in \Gamma/\Gamma(3)\Gamma_0$. 
The cardinality of the factor space $\Gamma/\Gamma(3)\Gamma_0$ is equal 
to $12,60,432$ for $n=2,3,4$ respectively in accordance with the discussion 
by Yoshida \cite{Yoshida 1998}, \cite{Yoshida 2001}, who gives a description
of this tessellation on the geometric side.

Two walls of $P$ are either orthogonal (with nonempty intersection in
$\mathbb{B}_r$) or parallel (with only intersection at an ideal point of 
$\overline{\mathbb{B}}_r$), and so $P$ is a right angled polytope. 
Equivalently, the Coxeter diagram of the chamber $P$ of the Coxeter group
$\Gamma_1$ has only edges with mark $\infty$. 
This Coxeter diagram (after deletion of all marks $\infty$) is of type 
$\mathrm{A}_3,\tilde{\mathrm{A}}_5$ for $n=2,3$ respectively, while for $n=4$ 
it is the Peterson graph, which we denote by $\mathrm{I}_{10}$.

Since $\Gamma/\Gamma(3)\cong\Gamma(2)/\Gamma(6)$ the group $\Gamma/\Gamma(3)$ is
generated by the cosets modulo $\Gamma(3)$ of a set of generators of $\Gamma(2)$.
Since $\Gamma(2)=\Gamma_1$ is a Coxeter group we take $r_i$ the reflections in the
walls of $P$ as Coxeter generators for $\Gamma(2)$ and hence $t_i=r_i\Gamma(3)$
are generators for $\Gamma/\Gamma(3)$. Because the $r_i$ are reflections the $t_i$
remain involutions in $\Gamma/\Gamma(3)$. Likewise if $r_i$ and $r_j$ commute
so do $t_i$ and $t_j$ commute. The relations between the $t_i$ in dimension $n$ 
are also valid in dimension $n+1$. In dimension $n=2$ it is easy to check that 
$t_it_jt_i=t_jt_it_j$ if the corresponding walls are parallel. Hence we recover
the Coxeter presentation of $S_4$. In all dimensions $2\leq n\leq7$ the group 
$\Gamma/\Gamma(3)$ becomes a factor group of the Coxeter group of the simply
laced Coxeter diagram obtained from that of $P$ by deletion of the marks $\infty$.
For $n=3$ this Coxeter diagram is the affine diagram of type $\tilde{\mathrm{A}}_5$
and it is easy to check that the translation lattice dies in $\Gamma/\Gamma(3)$.
This relation is also called deflation of the free hexagon, and we arrive at the 
following result. 

\begin{theorem}
For $2\leq n\leq7$ the group $\Gamma/\Gamma(3)$ is a factor group of $W/N$.
Here $W$ is the Coxeter group of the simply laced Coxeter diagram 
associated with $P$ and $W/N$ is the quotient by deflation of the free hexagons.
For $n\leq4$ we have in fact equality $\Gamma/\Gamma(3)=W/N$ and for $n=4$
we recover a presentation for $W(\mathrm{E}_6)$ found by Simons \cite{Simons 2005}. 
\end{theorem}

The fact that for $n=4$ these are a complete set of relations is an easy exercise
with the Petersen graph. The essential point of the theorem is to explain that 
this presentation has a natural geometric meaning from the action of 
$W(\mathrm{E}_6)$ on the moduli space $\overline{\mathcal{M}}_{rm}(cs)$
of marked maximally real cubic surfaces with its natural equivariant tessellation
as seen on the arithmetic side.

We do not know whether for $n=5,6,7$ the generators and relations given in the
theorem for $\Gamma/\Gamma(3)$ suffice to give a presentation. However this 
presentation for $W(\mathrm{E}_6)$ was found by Simons by analogy with similar 
presentations for the orthogonal group $\mathrm{PGO}^-_8(2)$ and the bimonster 
group $M\wr2$ as factor group of the Coxeter group on the incidence graph of 
the projective plane over a field of $2$ and $3$ elements by deflation of the free 
octagons and dodecagons respectively. This presentation of the bimonster was 
found by Conway and Simons \cite{Conway--Simons 2001} as a variation of the 
Ivanov--Norton theorem, which gives the bimonster group as a factor group 
of the Coxeter group $W(Y_{555})$ modulo the spider relation \cite{Ivanov 1992}, 
\cite{Norton 1992}. This presentation for $\mathrm{PGO}^-_8(2)$ and some of 
its subgroups (for example the Weyl group $W(\mathrm{E}_7)$) can be given
a similar geometric meaning. We would like to thank Masaaki Yoshida for
comments on an earlier version of this paper.
 
\section{The odd unimodular lattice $\mathbb{Z}^{n,1}$}

The odd unimodular lattice $\mathbb{Z}^{n,1}$ has basis $e_i$ for
$0\leq i\leq n$ with scalar product $(e_i,e_j)=\delta_{ij}$
for all $i,j$ except for $i=j=0$ in which case $e_0^2=-1$. 
The open set
\[ \mathbb{L}_r^{\times}=\{v\in\mathbb{R}^{n,1};v^2<0\} \]
has two connected components, and the component containing $e_0$ is
denoted by $\mathbb{L}_r^+$. The quotient space
\[ \mathbb{B}_r=\mathbb{L}_r^{\times}/\mathbb{R}^{\times}=
\mathbb{L}_r^+/\mathbb{R}^+ \]
is the real hyperbolic ball. The forward Lorentz group $\Or^+(\mathbb{R}^{n,1})$
is the index two subgroup of the full Lorentz group $\Or(\mathbb{R}^{n,1})$
preserving the component $\mathbb{L}_r^+$ and it acts faithfully on the ball 
$\mathbb{B}_r$. In addition
\[ \Gamma=\Or^+(\mathbb{Z}^{n,1})=
\Or^+(\mathbb{R}^{n,1})\cap\Or(\mathbb{Z}^{n,1}) \]
is a discrete subgroup of $\Or^+(\mathbb{R}^{n,1})$ acting on $\mathbb{B}_r$
properly discontinuously with cofinite volume. It contains reflections 
\[ s_{\alpha}(\lambda)=\lambda-2(\lambda,\alpha)\alpha/\alpha^2 \]
in roots $\alpha\in\mathbb{Z}^{n,1}$ of norm $1$ or norm $2$. Our notation
is $\alpha^2=(\alpha,\alpha)$ for the norm of $\alpha\in\mathbb{Z}^{n,1}$.
The next theorem is a (special case of a more general) result due to Vinberg 
\cite{Vinberg 1980} and for a pedestrian exposition of the proof we refer to 
the lecture notes on Coxeter groups by one of us \cite{Heckman 2013CG}. 

\begin{theorem}\label{Vinberg theorem}
For $2\leq n\leq 9$ the group $\Gamma=\Or^+(\mathbb{Z}^{n,1})$ is generated by
reflections $s_{\alpha}$ in roots $\alpha\in\mathbb{Z}^{n,1}$ of norm $1$ or norm $2$.
Moreover the Coxeter diagram of this reflection group $\Gamma$ is given by
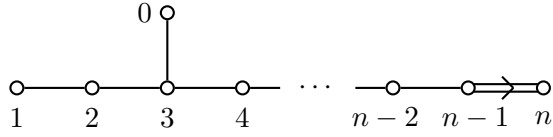
\begin{center}
\psset{unit=1mm}
\begin{pspicture}*(-50,-8)(50,15)
\pscircle(-20,10){1}
\pscircle(-40,0){1}
\pscircle(-30,0){1}
\pscircle(-20,0){1}
\pscircle(-10,0){1}
\pscircle(10,0){1}
\pscircle(20,0){1}
\pscircle(30,0){1}

\psline(-20,9)(-20,1)
\psline(-39,0)(-31,0)
\psline(-29,0)(-21,0)
\psline(-19,0)(-11,0)
\psline(-9,0)(-5,0)
\psline(5,0)(9,0)
\psline(11,0)(19,0)
\psline(20.7,0.5)(29.3,0.5)
\psline(20.7,-0.5)(29.3,-0.5)
\psline(24.5,1.5)(26,0)
\psline(24.5,-1.5)(26,0)

\rput(0,0){$\cdots$}
\rput(-23,10){$0$}
\rput(-40,-4){$1$}
\rput(-30,-4){$2$}
\rput(-20,-4){$3$}
\rput(-10,-4){$4$}
\rput(9,-4){$n-2$}
\rput(21,-4){$n-1$}
\rput(30,-4){$n$}
\end{pspicture}
\end{center}
with simple roots
\[ \alpha_0=e_0-e_1-e_2-e_3,
\alpha_1=e_1-e_2,\cdots,
\alpha_{n-1}=e_{n-1}-e_n,\alpha_n=e_n\;.  \]
For $n=2,3,4$ the Coxeter diagrams become
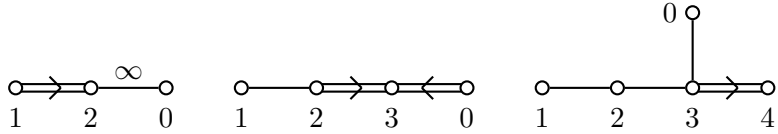
\begin{center}
\psset{unit=1mm}
\begin{pspicture}*(-55,-8)(55,12)
\pscircle(-50,0){1}
\pscircle(-40,0){1}
\pscircle(-30,0){1}

\pscircle(-20,0){1}
\pscircle(-10,0){1}
\pscircle(0,0){1}
\pscircle(10,0){1}

\pscircle(20,0){1}
\pscircle(30,0){1}
\pscircle(40,0){1}
\pscircle(50,0){1}
\pscircle(40,10){1}

\psline(-49.3,0.5)(-40.7,0.5)
\psline(-49.3,-0.5)(-40.7,-0.5)
\psline(-39,0)(-31,0)
\psline(-45.5,1.5)(-44,0)
\psline(-45.5,-1.5)(-44,0)

\psline(-19,0)(-11,0)
\psline(-9.3,0.5)(-0.7,0.5)
\psline(-9.3,-0.5)(-0.7,-0.5)
\psline(0.7,0.5)(9.3,0.5)
\psline(0.7,-0.5)(9.3,-0.5)
\psline(-5.5,1.5)(-4,0)
\psline(-5.5,-1.5)(-4,0)
\psline(5.5,1.5)(4,0)
\psline(5.5,-1.5)(4,0)

\psline(21,0)(29,0)
\psline(31,0)(39,0)
\psline(40,1)(40,9)
\psline(40.7,0.5)(49.3,0.5)
\psline(40.7,-0.5)(49.3,-0.5)
\psline(44.5,1.5)(46,0)
\psline(44.5,-1.5)(46,0)

\rput(-50,-4){$1$}
\rput(-40,-4){$2$}
\rput(-30,-4){$0$}
\rput(-35,2){$\infty$}

\rput(-20,-4){$1$}
\rput(-10,-4){$2$}
\rput(0,-4){$3$}
\rput(10,-4){$0$}

\rput(20,-4){$1$}
\rput(30,-4){$2$}
\rput(40,-4){$3$}
\rput(50,-4){$4$}
\rput(37,10){$0$}

\end{pspicture}
\end{center}
with $\alpha_0= \alpha_0=e_0-e_1-e_2$ a norm $1$
vector in case $n=2$.
\end{theorem}

The vertices of the closed fundamental chamber $D$ in $\overline{\mathbb{B}}_r$ 
are represented by the vectors (for $j=3,\cdots,n$)
\[ v_0=e_0,v_1=e_0-e_1,
v_2=2e_0-e_1-e_2,
v_j=3e_0-e_1-e_2-\cdots-e_j \]
as (anti)dual basis of the basis of simple roots. 
Let $D_0$ be the face of $D$ cut out by the long simple roots.
Hence $D_0$ is the edge of the triangle $D$ with vertices represented
by $v_0,v_2$ for $n=2$, while $D_0$ is the vertex of the simplex $D$
represented by $v_n$ for $3\leq n\leq 9$. 
Let $\Gamma_0$ be the subgroup of $\Gamma$ generated by the long simple roots, 
and so $\Gamma_0$ is the stabilizer of the face $D_0$. Clearly the group 
$\Gamma_0$ is a finite Coxeter group (of type $\mathrm{A}_1,
\mathrm{A}_1\sqcup\mathrm{A}_2,\mathrm{A}_4,\mathrm{D}_5,\mathrm{E}_6,
\mathrm{E}_7,\mathrm{E}_8$ respectively) for $2\leq n\leq 8$, which will be 
assumed from now on.
 
The convex polytope $P$ defined by
\[ P=\cup_{w\in\Gamma_0}\;wD \]
is the star of $D_0$, and will be called the Gosset polytope. 
The walls of $D$ which do not meet the relative interior of $D_0$ 
are cut out by the mirrors of the short simple roots. For $n=2$ there are
$2$ such edges of $D$ and for $3\leq n\leq 8$ there is just a unique such
wall of $D$. Hence the interior of $P$ is just a connected component of 
the complement of all mirrors in norm $1$ roots, and $P$ is a fundamental 
chamber for the normal subgroup $\Gamma_1$ of $\Gamma$ generated by the 
reflections in norm $1$ roots. Note that $\Gamma_1$ is in fact a subgroup
of the principal congruence subgroup $\Gamma(2)$ of $\Gamma$ of level $2$.
Because $\Gamma_0=\{w\in\Gamma;wP=P\}$ and the reflection group
$\Gamma_1$ is normal in $\Gamma$ and has $P$ as fundamental chamber we have 
the semidirect product decomposition $\Gamma=\Gamma_1\rtimes\Gamma_0$.

For $3\leq n\leq 8$ all walls of $P^n$ are congruent and of the form 
$P^{n-1}$. By induction on the dimension it can be shown that the 
set of vertices of $P$ consists of two orbits under $\Gamma_0$. One orbit
$\Gamma_0v_0$ are the actual vertices and the other orbit $\Gamma_0v_1$ are 
the ideal vertices of $P$. In turn this shows by a local analysis at $v_0$ 
and $v_1$ that all dihedral angles of $P$ inside $\mathbb{B}_r$ are $\pi/2$, 
and so $P$ is a right-angled polytope. Of course, at ideal vertices of $P$
the dihedral angle of intersecting walls can be $0$ as well. In other
words, the Coxeter diagram of the group of $\Gamma_1$ generated
by reflections in norm $1$ roots with fundamental chamber $P$ has 
only edges with mark $\infty$. The next result is due to Everitt, 
Ratcliffe and Tschantz \cite{Everitt--Ratcliffe--Tschantz 2010}.

\begin{theorem}\label{ERT theorem}
For $2\leq n\leq 7$ the group $\Gamma(2)$ is generated by reflections 
in norm $1$ roots, while for $n=8$ the subgroup of $\Gamma(2)$ generated 
by reflections in norm $1$ roots has index $2$.
\end{theorem}

\begin{proof}
Since $\Gamma=\Gamma_1\rtimes\Gamma_0$ we have to show that
$\Gamma_0\cap\Gamma(2)$ is the trivial group for $2\leq n\leq7$ and has
order $2$ for $n=8$. For $n=2$ the sublattice 
$L_0=\mathbb{Z}v_0+\mathbb{Z}v_2$ has discriminant $d=2$ while for 
$3\leq n\leq7$ the sublattice $L_0=\mathbb{Z}v_n$ has discriminant 
$d=9-n$. Hence the orthogonal complement $Q_0$ of $L_0$ in 
$\mathbb{Z}^{n,1}$ is just the root lattice of the finite Coxeter group
$\Gamma_0$ (of type $\mathrm{A}_1,\mathrm{A}_1\sqcup\mathrm{A}_2,
\mathrm{A}_4,\mathrm{D}_5,\mathrm{E}_6,\mathrm{E}_7,\mathrm{E}_8$ 
respectively). Indeed, that root lattice is contained in $Q_0$ and
has the correct discriminant $d$. The corresponding (rational) weight
lattice $P_0$, by definition the dual lattice of $Q_0$, is the orthogonal
projection of $\mathbb{Z}^{n,1}$ on $\mathbb{Q}\otimes Q_0$.

Now $w\in\Gamma_0$ also lies in $\Gamma(2)$ if and only if 
$w\lambda-\lambda\in2Q_0$ for all $\lambda\in P_0$. It is well known
that for $2\leq n\leq 7$ the set $\{\lambda\in P_0;\lambda^2<2\}$
is nonempty and spans $P_0$. For all these $\lambda$ the norm 
$(w\lambda-\lambda)^2$ is smaller than $8$ by the triangle inequality. 
But the only vector in $2Q_0$ of norm smaller than $8$ is the null 
vector. Hence $w=1$ and so $\Gamma_0\cap\Gamma(2)$ is the trivial group.
For $n=8$ the elements of minimal positive norm in the lattice
$P_0=Q_0$ of type $\mathrm{E}_8$ form the root system $R(\mathrm{E}_8)$ 
of type $E_8$ of vectors of norm $2$. If $(w-1)\alpha\in2Q_0$ for 
$w\in\Gamma_0$ and $\alpha\in R(\mathrm{E}_8)$ then
either $(w-1)\alpha$ has norm smaller than $8$ and $w\alpha=\alpha$, 
or $(w-1)\alpha$ has norm $8$ and $w\alpha=-\alpha$. 
If $w\alpha=\pm\alpha$ for all $\alpha\in R(\mathrm{E}_8)$ then one
easily concludes that $w=\pm1$. Hence $\Gamma_0\cap\Gamma(2)=\{\pm1\}$
has order $2$ for $n=8$.
\end{proof}

For $n=2,3,4$ the Coxeter diagram of the reflection group $\Gamma_1=\Gamma(2)$
has the following explicit description.

\begin{theorem}\label{Coxeter diagram theorem}
The Coxeter diagrams of $\Gamma$ on the left and of $\Gamma(2)$ on the right 
are given by 
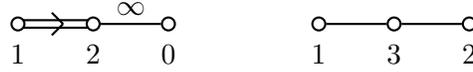
\begin{center}
\psset{unit=1mm}
\begin{pspicture}*(-60,-8)(20,8)
\pscircle(-50,0){1}
\pscircle(-40,0){1}
\pscircle(-30,0){1}

\pscircle(-10,0){1}
\pscircle(0,0){1}
\pscircle(10,0){1}

\psline(-49.3,0.5)(-40.7,0.5)
\psline(-49.3,-0.5)(-40.7,-0.5)
\psline(-39,0)(-31,0)
\psline(-45.5,1.5)(-44,0)
\psline(-45.5,-1.5)(-44,0)

\psline(-9,0)(-1,0)
\psline(1,0)(9,0)

\rput(-50,-4){$1$}
\rput(-40,-4){$2$}
\rput(-30,-4){$0$}
\rput(-35,2){$\infty$}

\rput(-10,-4){$1$}
\rput(0,-4){$3$}
\rput(10,-4){$2$}

\end{pspicture}
\end{center}
for $n=2$, and
\begin{center}
\psset{unit=1mm}
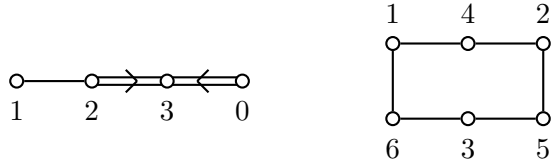
\begin{pspicture}*(-25,-12)(65,12)

\pscircle(-20,0){1}
\pscircle(-10,0){1}
\pscircle(0,0){1}
\pscircle(10,0){1}

\pscircle(30,5){1}
\pscircle(40,5){1}
\pscircle(50,5){1}
\pscircle(30,-5){1}
\pscircle(40,-5){1}
\pscircle(50,-5){1}

\psline(-19,0)(-11,0)
\psline(-9.3,0.5)(-0.7,0.5)
\psline(-9.3,-0.5)(-0.7,-0.5)
\psline(0.7,0.5)(9.3,0.5)
\psline(0.7,-0.5)(9.3,-0.5)
\psline(-5.5,1.5)(-4,0)
\psline(-5.5,-1.5)(-4,0)
\psline(5.5,1.5)(4,0)
\psline(5.5,-1.5)(4,0)

\psline(31,5)(39,5)
\psline(41,5)(49,5)
\psline(50,4)(50,-4)
\psline(30,4)(30,-4)
\psline(31,-5)(39,-5)
\psline(41,-5)(49,-5)

\rput(-20,-4){$1$}
\rput(-10,-4){$2$}
\rput(0,-4){$3$}
\rput(10,-4){$0$}

\rput(30,9){$1$}
\rput(40,9){$4$}
\rput(50,9){$2$}
\rput(50,-9){$5$}
\rput(40,-9){$3$}
\rput(30,-9){$6$}

\end{pspicture}
\end{center}
for $n=3$, and 
\begin{center}
\psset{unit=1mm}
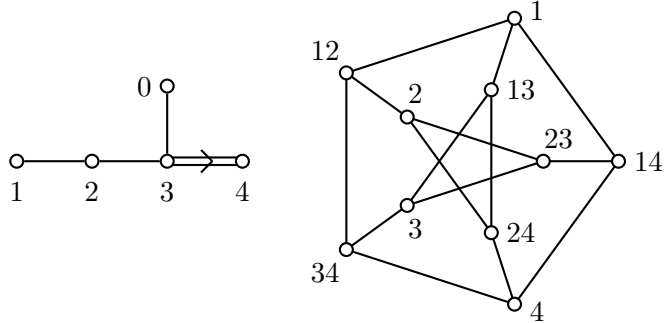
\begin{pspicture}*(-70,-22)(30,22)

\pscircle(-60,0){1}
\pscircle(-50,0){1}
\pscircle(-40,0){1}
\pscircle(-30,0){1}
\pscircle(-40,10){1}

\psline(-51,0)(-59,0)
\psline(-41,0)(-49,0)
\psline(-40,1)(-40,9)
\psline(-30.7,0.5)(-39.3,0.5)
\psline(-30.7,-0.5)(-39.3,-0.5)
\psline(-35.5,1.5)(-34,0)
\psline(-35.5,-1.5)(-34,0)

\rput(-60,-4){$1$}
\rput(-50,-4){$2$}
\rput(-40,-4){$3$}
\rput(-30,-4){$4$}
\rput(-43,10){$0$}

\psline(20,0)(6.18,19.02)
\psline(6.18,19.02)(-16.18,11.755)
\psline(-16.18,11.755)(-16.18,-11.755)
\psline(-16.18,-11.755)(6.18,-19.02)
\psline(6.18,-19.02)(20,0)

\psline(10,0)(20,0)
\psline(3.09,9.51)(6.18,19.02)
\psline(-8.09,5.88)(-16.18,11.755)
\psline(-8.09,-5.88)(-16.18,-11.755)
\psline(3.09,-9.51)(6.18,-19.02)

\psline(10,0)(-8.09,5.88)
\psline(-8.09,5.88)(3.09,-9.51)
\psline(3.09,-9.51)(3.09,9.51)
\psline(3.09,9.51)(-8.09,-5.88)
\psline(-8.09,-5.88)(10,0)

\pscircle[fillstyle=solid,fillcolor=white](10,0){1}
\pscircle[fillstyle=solid,fillcolor=white](20,0){1}
\pscircle[fillstyle=solid,fillcolor=white](3.09,9.51){1}
\pscircle[fillstyle=solid,fillcolor=white](6.18,19.02){1}
\pscircle[fillstyle=solid,fillcolor=white](-8.09,5.88){1}
\pscircle[fillstyle=solid,fillcolor=white](-16.18,11.755){1}
\pscircle[fillstyle=solid,fillcolor=white](-8.09,-5.88){1}
\pscircle[fillstyle=solid,fillcolor=white](-16.18,-11.755){1}
\pscircle[fillstyle=solid,fillcolor=white](3.09,-9.51){1}
\pscircle[fillstyle=solid,fillcolor=white](6.18,-19.02){1}

\rput(9.18,20.02){$1$}
\rput(-7,9){$2$}
\rput(-7,-9){$3$}
\rput(9.18,-20.02){$4$}
\rput(-19,14.755){$12$}
\rput(7,9.51){$13$}
\rput(24,0){$14$}
\rput(12,3){$23$}
\rput(7,-9.51){$24$}
\rput(-19,-14.755){$34$}

\end{pspicture}
\end{center}
for $n=4$ respectively. All edges of the Coxeter diagrams of $\Gamma(2)$ 
have mark $\infty$, but for simplicity and because of the next theorem
these are left out in the drawn diagrams. 
The last diagram for $n=4$ with $10$ nodes is the so called 
Petersen graph and will be denoted $\mathrm{I}_{10}$.
The automorphism groups $\Gamma_0\cong\Gamma/\Gamma(2)$ of these Coxeter
diagrams of $\Gamma(2)$ are equal to $S_2,S_2\times S_3,S_5$ as the Weyl 
groups of type $\mathrm{A}_1,\mathrm{A}_1\sqcup\mathrm{A}_2,\mathrm{A}_4$
respectively. 
\end{theorem}

\begin{proof}
Let $s_i$ for $i=0,1,\cdots,n$ be the simple reflections of the group 
$\Gamma$ as numbered in Theorem{\;\ref{Vinberg theorem}}. We shall treat 
the cases $n=2,3,4$ separately.

For $n=2$ the fundamental domain $D$ is a hyperbolic triangle with 
angles $\{\pi/4,0,\pi/2\}$ at the vertices $v_0,v_1,v_2$ respectively. 
The Gosset polytope $P=D\cup s_1D$ is a hyperbolic triangle with angles 
$\{\pi/2,0,0\}$ at the vertices $v_0,v_1,s_1v_1$. It is a fundamental 
domain for the action of the Coxeter group $\Gamma(2)$ with simple 
generators 
\[ r_1=s_1s_2s_1,r_2=s_2,r_3=s_0 \]
whose Coxeter diagram is the $\mathrm{A}_3$ diagram with marks $\infty$ 
on the edges rather than the usual mark $3$.

For $n=3$ the Gosset polytope $P$ is a double tetrahedron $P=T\cup s_0T$ 
with hyperbolic tetrahedron $T$ the union over $wD$ with 
$w\in S_3=\langle s_1,s_2\rangle$ and $\{v_0,v_1,s_1v_1,s_2s_1v_1\}$ as 
the set of vertices. The Coxeter diagram of $T$ is the $\mathrm{D}_4$ 
diagram with marks $4$ on the edges rather than the usual mark $3$. 
The reflection $s_0$ corresponds to the central node, and the reflections 
\[ r_1=s_1r_2s_1,r_2=s_2s_3s_2,r_3=s_3 \]
correspond to the three extremal nodes. The polytope $P$ is the fundamental 
domain for the action of the Coxeter group $\Gamma(2)$ with simple 
generators 
\[ r_1=s_1r_2s_1,r_2=s_2s_3s_2,r_3=s_3,
r_4=s_0r_3s_0,r_5=s_0r_1s_0,r_6=s_0r_2s_0 \]
whose Coxeter diagram is the $\tilde{\mathrm{A}}_5$ diagram with marks 
$\infty$ on the edges rather than the usual mark $3$. 

For $n=4$ the Gosset polytope $P$ is the union $\cup_w wD$ over 
$w\in\Gamma_0$ with $\Gamma_0=S_5$ the group generated by the 
reflections $s_0,s_1,s_2,s_3$ in the long simple roots. The vertex 
$v_4$ of $D$ is interior point of $P$ and $\Gamma_0$ is the symmetry 
group of $P$ generated by the reflections in the mirrors through $v_4$. 
The group $\Gamma(2)$ is generated by the simple reflections 
\[ r_i=ws_4w^{-1} \]
with $w\in S_5$ and $i\in I=S_5/(S_2\times S_3)$ the left coset of $w$
for the centralizer of $s_4$ in $S_5$, which is just generated by 
$s_0,s_1,s_2$. The cardinality of $I$ is equal to $10$ and the Coxeter
diagram of $P$ is the Petersen graph $\mathrm{I}_{10}$, but with the edges 
marked $\infty$ rather than $3$. Indeed, by Theorem{\;\ref{Vinberg theorem}}
\[ \alpha_0=e_0-e_1-e_2-e_3,
\alpha_1=e_1-e_2,\alpha_2=e_2-e_3,
\alpha_3=e_3-e_4,\alpha_4=e_4 \]
is the basis of simple roots for $D$. 
Hence both $\beta_3=s_3(\alpha_4)=e_3$
and $\beta_{12}=s_0(\beta_3)=e_0-e_1-e_2$
are simple roots for $P$. Using the action of $\langle s_1,s_2,s_3\rangle$ 
we see that
\[ \beta_i=e_i,\beta_{jk}=e_0-e_j-e_k \]
are simple roots of $P$ for $1\leq i\leq4$ and $1\leq j<k\leq4$. 
Because $P$ has $10$ simple roots these are all simple roots of $P$.
The Gosset polytope $P$ has $5$ actual vertices, which are the 
transforms under $\Gamma_0$ of $v_0$. Likewise it has $5$ ideal 
vertices, which are the transforms under $\Gamma_0$ of the cusp 
$v_1$ of $D$. 
\end{proof}

The Petersen graph was described by Petersen in 1898 \cite{Petersen 1898},
but was in fact discovered before in 1886 by Kempe \cite{Kempe 1886}.

\begin{theorem}\label{odd presentation theorem}
Let $\Gamma=\Or^+(\mathbb{Z}^{n,1})$ and let $\Gamma(2)$ and $\Gamma(3)$ 
be the principal congruence subgroups of level $2$ and level $3$ 
respectively for $n=2,3,4$. Then the group $\Gamma/\Gamma(3)$ is equal to 
\[ \PGO_3(3)=S_4=W(\mathrm{A}_3),\PGO_4(3)=S_6=W(\mathrm{A}_5),
\PGO_5(3)=W(\mathrm{E}_6) \]
respectively. If we denote by $r_i$ the Coxeter generators of $\Gamma(2)$ 
in the notation of Theorem{\;\ref{Coxeter diagram theorem}} 
then $t_i=r_i\Gamma(3)$ are generators for $\Gamma/\Gamma(3)$.
In fact $\Gamma/\Gamma(3)$ has a presentation with generators the involutions 
$t_i$ and with braid and deflation relations. The braid relations amount to
\[ t_it_j=t_jt_i\;,\;t_it_jt_i=t_jt_it_j \]
if the nodes with index $i$ and $j$ are disconnected and connected respectively, 
and so $\Gamma/\Gamma(3)$ is a factor group of the Coxeter group associated to 
the simply laced Coxeter diagrams 
$\mathrm{A}_3,\tilde{\mathrm{A}}_5,\mathrm{P}_{10}$
of Theorem{\;\ref{Coxeter diagram theorem}}. The deflation relations mean 
that for each subdiagram of type $\tilde{\mathrm{A}}_5$, also called 
a free hexagon, the translation lattice of the affine Coxeter group 
$W(\tilde{\mathrm{A}}_5)$ dies in $\Gamma/\Gamma(3)$.
\end{theorem}

\begin{proof}
It is well known that $\mathrm{PGO}_{n+1}$ is equal to
$W(\mathrm{A}_3),W(\mathrm{A}_5),W(\mathrm{E}_6)$ for $n=2,3,4$ 
respectively \cite{ATLAS 1985}.
Clearly $\Gamma/\Gamma(3)\cong\Gamma(2)/\Gamma(6)$, and so 
$\Gamma/\Gamma(3)$ is a factor group of the Coxeter group $\Gamma(2)$
with Coxeter diagram given by Theorem{\;\ref{Coxeter diagram theorem}}
with all edges marked $\infty$.

If $\alpha,\beta\in\mathbb{Z}^{n,1}$ are norm $1$ roots with 
$(\alpha,\beta)=-1$ then a straightforward computation yields
\[ (s_{\beta}s_{\alpha}s_{\beta}-s_{\alpha}s_{\beta}s_{\alpha})\lambda=
6(\lambda,\alpha)\alpha-6(\lambda,\beta)\beta \]
for all $\lambda\in\mathbb{Z}^{n,1}$, which in turn implies 
$s_{\beta}s_{\alpha}s_{\beta}\equiv s_{\alpha}s_{\beta}s_{\alpha}$
modulo $\Gamma(3)$. Hence $\Gamma/\Gamma(3)$ is a factor group of
the Coxeter group with the simply laced Coxeter diagrams of 
Theorem{\;\ref{Coxeter diagram theorem}}, because the marks $\infty$  
become a $3$ and are deleted. For $n=2$ we recover the Coxeter 
presentation of $S_4=W(\mathrm{A}_3)$.

For $n=3$ the group $\Gamma/\Gamma(3)=S_6$ is the factor group of the 
affine Coxeter group $W(\tilde{\mathrm{A}}_5)$ by its translation lattice.
Indeed, in the notation of Theorem{\;\ref{Coxeter diagram theorem}} and 
its proof we have
\[ r_1=s_{e_1},r_2=s_{e_2},r_3=s_{e_3},
r_4=s_{e_0-e_1-e_2},
r_5=s_{e_0-e_2-e_3},
r_6=s_{e_0-e_1-e_3} \]
and the relation
\[ t_1t_4t_2t_5t_3t_6t_3t_5t_2t_4=1 \]
in $\Gamma/\Gamma(3)$ follows by direct inspection. Since the element on 
the left side in the affine Coxeter group $W(\tilde{\mathrm{A}}_5)$ is a 
translation over a coroot this shows that the translation lattice dies 
in $\Gamma/\Gamma(3)$. This relation is also called deflation of the 
free hexagon.

For $n=4$ we recover a presentation for the group $W(\mathrm{E_6})$
as found by Christopher Simons \cite{Simons 2005}.
It is the factor group of the Coxeter group $W(\mathrm{P}_{10})$ of
the Petersen graph $\mathrm{P}_{10}$ by deflation of all free hexagons.
This somewhat odd presentation for $W(\mathrm{E}_6)$ can be seen in 
the usual $\mathrm{E}_6$ diagram 
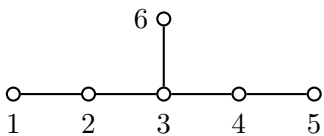
\begin{center}
\psset{unit=1mm}
\begin{pspicture}*(-30,-8)(30,15)

\pscircle(-20,0){1}
\pscircle(-10,0){1}
\pscircle(0,0){1}
\pscircle(10,0){1}
\pscircle(20,0){1}
\pscircle(0,10){1}

\psline(-19,0)(-11,0)
\psline(-9,0)(-1,0)
\psline(1,0)(9,0)
\psline(11,0)(19,0)
\psline(0,1)(0,9)

\rput(-3,10){$6$}
\rput(-20,-4){$1$}
\rput(-10,-4){$2$}
\rput(0,-4){$3$}
\rput(10,-4){$4$}
\rput(20,-4){$5$}

\end{pspicture}
\end{center}
as follows. The group generated by the simple reflections $s_i$ for
$1\leq i\leq5$ is the symmetric group $S_6$. The orbit under the symmetric
group $S_5$ generated by $s_i$ for $1\leq i\leq4$ of the root $\alpha_6$
has cardinality $10$ and the reflections in these $10$ roots generate the
Weyl group $W(\mathrm{D}_5)$ generated by the reflections $s_1,s_2,s_3,s_4,s_6$.
However $S_6$ has an outer automorphism \cite{Todd 1945}, and the image of $S_5$ under this
automorphism is denoted $\tilde{S}_5$. The orbit under the twisted $\tilde{S}_5$
of the root $\alpha_6$ has again cardinality $10$, and the Gram matrix of this 
set of $10$ roots is the incidence matrix of the Petersen graph, so 
$(\alpha,\beta)=0,1,2$ if $\alpha$ and $\beta$ are disconnected, or are 
connected by an edge, or are equal respectively. 

An explicit way of understanding that a set of $10$ vectors with such a
Gram matrix exists in the root system $R(\mathrm{E}_6)$ goes as follows. 
Denote by $\{\alpha_j\}$ the basis of simple roots of $R(\mathrm{E}_6)$ 
numbered as in the above diagram. Then we take
\[ \beta_{13}=-\alpha_1,\beta_1=\alpha_2,\beta_{14}=-\alpha_3,
\beta_4=\alpha_4,\beta_{34}=-\alpha_5,\beta_{23}=\alpha_6 \]
in the numbering of nodes of $\mathrm{P}_{10}$ as in 
Theorem{\;\ref{Coxeter diagram theorem}}. In turn this implies
\begin{eqnarray*}
\beta_3 &=& -\alpha_1-\alpha_2-\alpha_3-\alpha_4-\alpha_5 \\
\beta_{24} &=& \alpha_2+2\alpha_3+2\alpha_4+\alpha_5+\alpha_6 \\
\beta_2 &=& \alpha_1+2\alpha_2+3\alpha_3+2\alpha_4+\alpha_5+2\alpha_6 \\
\beta_{12} &=& \alpha_1+2\alpha_2+2\alpha_3+\alpha_4+\alpha_6
\end{eqnarray*}
by looking for suitable free hexagons, as the alternating sum
of the roots of a free hexagon vanishes. Hence we recover the presentation 
of Simons for the Weyl group $W(\mathrm{E}_6)$ as the quotient of the 
Coxeter group $W(\mathrm{P}_{10})$ by deflation of all free hexagons.
\end{proof}

\begin{remark}
The automorphism group $S_5$ of the Petersen graph can be identified with 
the group of geometric automorphisms of the Clebsch diagonal surface
\[ u+v+w+x+y=0\;,\;u^3+v^3+w^3+x^3+y^3=0 \]
in projective three space. Via the period map this surface corresponds 
to the central point 
$v_4=3e_0-e_1-e_2-e_3-e_4$
of the Gosset polytope $P$ for $n=4$. In this way $S_5$ becomes a subgroup
of $W(\mathrm{E}_6)$ as symmetry group of the configuration of the $27$
lines on the Clebsch diagonal surface. This monomorphism
$S_5\hookrightarrow W(\mathrm{E}_6)$, as described in the above proof, was
already discussed by Segre \cite{Segre 1942}. 

Likewise the dihedral group $D_6$ of order $12$ as automorphism group of the free 
hexagon can be identified with the group of geometric automorphisms of the degree 
$6$ binary form $u^6+v^6$, which corresponds via the period map to the central 
point $v_3=3e_0-e_1-e_2-e_3$ of the Gosset
polytope $P$ for $n=3$. In this way $D_6\hookrightarrow S_6$ and up to conjugation 
by (inner and outer) automorphisms of $S_6$ there is a unique monomorphism 
$D_6\hookrightarrow S_6$.

The symmetric group $S_2$ as automorphism group of the Coxeter diagram
$\mathrm{A}_3$ can be identified with the group of geometric automorphisms of the one 
parameter family of degree $6$ binary forms $(u+v)^2(u^4+tu^2v^2+v^4)$ with $-2<t<2$ 
via $(u,v)\mapsto(v,u)$, which corresponds via the period map to the central
line segment between the vertices $v_0$ and $v_2$ inside the Gosset polytope $P$ 
for $n=2$. In this way $S_2\hookrightarrow V_4\hookrightarrow S_4$ and up to 
conjugation there is a unique such monomorphism.
\end{remark}

Via the period map isomorphism $\overline{\mathcal{M}}_{rm}\rightarrow\overline{\mathbb{B}}_r/\Gamma(3)$ 
we get a tessellation of the moduli space $\overline{\mathcal{M}}_{rm}$ 
of marked maximally real objects by congruent copies $\gamma P$ of the Gosset 
polytope with $\gamma$ in the factor space $\Gamma/\Gamma(3)\Gamma_0$ and 
$\Gamma_0=\mathrm{Aut}(P)\hookrightarrow\Gamma/\Gamma(3)$ 
the natural monomorphism. The glue prescription is given by 
\[ \overline{\mathbb{B}}_r/\Gamma(3)=\{\sqcup_{\gamma}\;\gamma P\}/\sim \]
with
\[ \gamma P\supset \gamma F_i\sim(\gamma t_i)F_i\subset (\gamma t_i)P \]
and $F_i$ the wall of $P$ fixed by $r_i$ in the notation of 
Theorem{\;\ref{odd presentation theorem}}. The glue prescription was discussed 
in geometric terms by Yoshida \cite{Yoshida 1998},\cite{Yoshida 2001}.
This paper grew out of an attempt to understand his work.

\noindent
Gert Heckman, Radboud University Nijmegen: g.heckman@math.ru.nl
\newline
Sander Rieken, Radboud University Nijmegen: s.rieken@math.ru.nl

\end{document}